\documentclass{amsart}
\usepackage{amssymb}
\usepackage{graphicx}

\theoremstyle{plain}
\newtheorem{theorem}{Theorem}[section]
\newtheorem{corollary}[theorem]{Corollary}

\newtheorem{proposition}[theorem]{Proposition}
\newtheorem{lemma}[theorem]{Lemma}
\newtheorem{conjecture}[theorem]{Conjecture}
\numberwithin{equation}{section}

\begin{document}

\title[Sums of the floor function]{Sums of the floor function related to class 
numbers of imaginary quadratic fields}

\author{Marc Chamberland}
\address{Department of Mathematics and Statistics, Grinnell College,
Grinnell, IA 50112, USA}
\email{chamberl@math.grinnell.edu}
\author{Karl Dilcher}
\address{Department of Mathematics and Statistics,
         Dalhousie University,
         Halifax, Nova Scotia, B3H 4R2, Canada}
\email{dilcher@mathstat.dal.ca}
\keywords{Floor function, quadratic residue, class number, imaginary quadratic
field}
\subjclass[2010]{Primary 11A25; Secondary 11R11, 11R29}
\thanks{Research supported in part by the Natural Sciences and Engineering
        Research Council of Canada, Grant \# 145628481}

\date{}

\setcounter{equation}{0}

\begin{abstract}
A curious identity of Bunyakovsky (1882), made more widely known by P\'olya and
Szeg{\H o} in their ``Problems and Theorems in Analysis", gives an evaluation
of a sum of the floor function of square roots involving primes
$p\equiv 1\pmod{4}$. We evaluate this sum also in the case $p\equiv 3\pmod{4}$,
obtaining an identity in terms of the class number of the imaginary quadratic
field ${\mathbb Q}(\sqrt{-p})$. We also consider certain cases where the prime
$p$ is replaced by a composite integer. Class numbers of imaginary quadratic
fields are again involved in some cases.

\end{abstract}

\maketitle

\section{Introduction}\label{sec:1}

The {\it floor function}, also known as the {\it greatest integer function\/}, 
of $x\in{\mathbb R}$ is the integer $\lfloor{x}\rfloor$ defined by
$\lfloor{x}\rfloor\leq x < \lfloor{x}\rfloor+1$.
Sums involving the floor function have a long history and have been extensively
studied. One of the best known sums of this type is
\begin{equation}\label{1.1}
\lfloor{x}\rfloor+\left\lfloor{x+\tfrac{1}{n}}\right\rfloor+
\left\lfloor{x+\tfrac{2}{n}}\right\rfloor
+\cdots+\left\lfloor{x+\tfrac{n-1}{n}}\right\rfloor = \lfloor{nx}\rfloor,
\end{equation}
where $n$ is a positive integer. The identity \eqref{1.1} is
due to Hermite \cite[p.~315]{He}; see also \cite[p.~85]{GKP} or 
\cite[Part~8, Problem~9]{PS2}.

A different type of such a sum, involving square roots, is
\begin{equation}\label{1.2}
1+\lfloor{\sqrt{2}}\rfloor+\lfloor{\sqrt{3}}\rfloor
+\cdots+\lfloor{\sqrt{n-1}}\rfloor = na-\tfrac{1}{3}a(a+\tfrac{1}{2})(a+1),
\end{equation}
where $a=\lfloor{\sqrt{n}}\rfloor$; see \cite[p.~87]{GKP}. A somewhat related
identity can be found in the famous two-volume ``Problems and Theorems in
Analysis" by P\'olya and Szeg{\H o} \cite{PS2} as Problem~20 in Part~8: if 
$p$ is a prime of the form $4n+1$, then
\begin{equation}\label{1.3}
\left\lfloor{\sqrt{p}}\right\rfloor+\left\lfloor{\sqrt{2p}}\right\rfloor+
\left\lfloor{\sqrt{3p}}\right\rfloor
+\cdots+\left\lfloor{\sqrt{\tfrac{p-1}{4}p}}\right\rfloor = \frac{p^2-1}{12}.
\end{equation}
P\'olya and Szeg{\H o} attribute this identity to Bouniakowski \cite{Bou} 
(better known as V.~Ya.~Bunyakovsky), and
in \cite{PS2} it is the last of several exercises that use the technique of
counting lattice points. A different method was used by Shirali in \cite{Sh}, 
which has been the main inspiration for the current paper. 

The identity \eqref{1.3} leads us to consider the arithmetic function
\begin{equation}\label{1.4}
f(n):=\sum_{j=1}^{\lfloor{n/4}\rfloor}\left\lfloor{\sqrt{jn}}\right\rfloor
-\frac{n^2-1}{12}\qquad(n\in{\mathbb N}).
\end{equation}
We can therefore rewrite \eqref{1.3} as follows.

\begin{proposition}[Bouniakowski, 1882]\label{prop:1.1}
For any  prime $p\equiv1\pmod{4}$ we have
\begin{equation}\label{1.5}
f(p)=0.
\end{equation}
\end{proposition}

The function $f(n)$ is the main object of study in this paper, and \eqref{1.5}
leads to a first obvious question: what can we say about $f(p)$ for primes
$p\equiv 3\pmod{4}$?

To explore this question, we list $f(p)$ in Table~1 for all such primes $p<100$,
with the exception of $p=3$. The regularity of the first three cases might lead
to the idea of an easy linear transformation to map them to 0.
This is done in columns 3 and 7 of Table~1. We now see that we have
$-p-1-4f(p)=0$ exactly when $p=7$, 11, 19, 43, and 67. If this sequence doesn't
already look familiar, the OEIS \cite{OEIS} reveals that these numbers represent
most of the square-free positive integers $d$ for which the ring of integers of
the imaginary quadratic field ${\mathbb Q}(\sqrt{-d})$ has unique factorization.
It is well-known that the largest such number is the prime $p=163$; and indeed,
we compute $f(p)=-41$ and $-p-1-4f(p)=0$. In fact, Shirali \cite[p.~270]{Sh} 
already observed this connection in a slightly different but equivalent 
setting. Since unique factorization means that the class number of the 
corresponding quadratic field is 1, it makes sense to consider the class 
numbers $h(-p)$ of ${\mathbb Q}(\sqrt{-p})$ for all $p\equiv 3\pmod{4}$ and
$p\geq 7$ as listed, for instance, in \cite[p.~425]{BS}. See columns 4 and 8
in Table~1.

\bigskip
\begin{center}
\begin{tabular}{|c|c|c|c||c|c|c|c|}
\hline
$p$ & $f(p)$ & $-p-1-4f(p)$ & $h(-p)$ & $p$ & $f(p)$ & $-p-1-4f(p)$ & $h(-p)$\\
\hline
7 & $-2$ & 0 & 1 & 47 & $-14$ & 8 & 5 \\
11 & $-3$ & 0 & 1 & 59 & $-16$ & 4 & 3 \\
19 & $-5$ & 0 & 1 & 67 & $-17$ & 0 & 1 \\
23 & $-7$ & 4 & 3 & 71 & $-21$ & 12 & 7 \\
31 & $-9$ & 4 & 3 & 79 & $-22$ & 8 & 5 \\
43 & $-11$& 0 & 1 & 83 & $-22$ & 4 & 3 \\
\hline
\end{tabular}

\medskip
{\bf Table~1}: $f(p)$ and $h(-p)$ for primes $p\equiv$ 3 (mod 4), $7\leq p<100$.
\end{center}
\bigskip

It is now quite clear how columns 3, 4 and 7, 8 relate to each other. In fact,
we can state the following result.

\begin{proposition}\label{prop:1.2}
Let $p\geq 7$ be a prime with $p\equiv3\pmod{4}$. Then
\begin{equation}\label{1.6}
f(p) = \frac{1}{4}\big(1-p-2h(-p)\big),
\end{equation}
where $h(-p)$ is the class number of ${\mathbb Q}(\sqrt{-p})$.
\end{proposition}

It is the purpose of this paper to prove Proposition~\ref{prop:1.2} and other
identities of this type. One main tool is the connection between $f(n)$
and quadratic residues that was established by Shirali \cite{Sh} in a special
case. We generalize this in Section~\ref{sec:2} to arbitrary $n\in{\mathbb N}$
and obtain some initial elementary results in Section~\ref{sec:3}.

A second important tool is a class number formula of Dirichlet. We introduce
it in Section~\ref{sec:4}, where we also establish or quote specific
versions used in this paper. All this is then applied in Section~\ref{sec:5}
to obtain identities for $f(2p)$ and $f(4p)$, where $p$ is an odd prime, and
in Section~\ref{sec:6} to arbitrary powers of all primes. We finish with some
conjectures in Section~\ref{sec:7}.

\section{Connections with quadratic residues}\label{sec:2}

In the process of providing an alternative proof of Bouniakowski's result
\eqref{1.5}, Shirali \cite{Sh} established a connection with quadratic residues.
We extend this approach and set the stage for a connection with class
number formulas. Following \cite{Sh}, we set
\begin{equation}\label{2.1}
F(n):=\sum_{j=1}^{\lfloor{n/4}\rfloor}\left\lfloor{\sqrt{jn}}\right\rfloor
\qquad(n\in{\mathbb N}),
\end{equation}
and for integers $a\geq 0$, $b\geq 1$ we denote by Rem$(a\div b)$ the smallest 
nonnegative remainder as $a$ is divided by $b$. We are now ready to state and
prove the main connection between $F(n)$ and quadratic residues.
In what follows, we always assume that $n$ is a positive integer, and we write
\begin{equation}\label{2.1a}
n = 4\nu+r,\quad0\leq r\leq 3,\quad\hbox{so that}\quad
\textstyle{\nu=\lfloor\frac{n}{4}\rfloor.}
\end{equation}
We will also use the fact that any nonzero integer has a unique representation
as a product of a squarefree integer and a perfect square; see, e.g., 
\cite[p.~29]{NZM}.

\begin{proposition}\label{prop:2.1}
$(a)$ With $n$ as in \eqref{2.1a}, let $1\leq r\leq 3$ and $n=P_n\cdot Q_n^2$,
with $P_n$ squarefree. Then
\begin{equation}\label{2.2}
F(n)=2\nu^2-\frac{\nu}{3n}\cdot\left(8\nu^2+6\nu+1\right)+\frac{1}{2}(Q_n-1)
+\frac{1}{n}\sum_{k=1}^{2\nu}{\rm Rem}(k^2\div n).
\end{equation}
$(b)$ If $n=4\nu$, let $\nu=\overline{P}_n\cdot\overline{Q}_n^2$, with 
$\overline{P}_n$ squarefree. Then
\begin{equation}\label{2.3}
F(4\nu)=\tfrac{4}{3}\nu^2-\tfrac{1}{2}\nu-\tfrac{1}{12}+\overline{Q}_n
+\frac{1}{4\nu}\sum_{k=1}^{2\nu}{\rm Rem}(k^2\div 4\nu).
\end{equation}
\end{proposition}

Proposition~\ref{prop:2.1}(b) simplifies if we consider the function $f(n)$ 
defined by \eqref{1.4}. The following can be obtained by straightforward
manipulation.

\begin{corollary}\label{cor:2.2}
Let $n\equiv 0\pmod{4}$ be a positive integer. Then
\begin{equation}\label{2.3a}
f(n)=-\frac{n}{8}+\overline{Q}_n
+\frac{1}{n}\sum_{k=1}^{n/2}{\rm Rem}(k^2\div n),
\end{equation}
where $\frac{n}{4}=\overline{P}_n\cdot\overline{Q}_n^2$, with $\overline{P}_n$ 
squarefree.
\end{corollary}

To prove Proposition~\ref{prop:2.1}, we require the following lemma.

\begin{lemma}\label{lem:2.2}
Let $n\geq 1$ be an integer, and denote
\begin{equation}\label{2.4}
A(n):=\left|\left\{k: 1\leq k\leq \textstyle{2\lfloor\frac{n}{4}\rfloor-1}, 
n\mid k^2
\right\}\right|.
\end{equation}
Then the following is true:
\begin{enumerate}
\item[(a)] If $4\nmid n$ and $n=P\cdot Q^2$, where $P$ is squarefree, then
$A(n)=\frac{1}{2}(Q-1)$.
\item[(b)] If $n=4\cdot P\cdot Q^2$, where $P$ is squarefree, then $A(n)=Q-1$.
\end{enumerate}
\end{lemma}

\begin{proof}
(a) The condition $n\mid k^2$ is equivalent to $k^2=mPQ^2$ for some integer
$m\geq 1$. Since $P$ is squarefree, the fact that $P$ divides $k^2$ implies
$P\mid k$, and thus $P\mid m$ as well. We therefore get
\begin{equation}\label{2.5}
\left(\frac{k}{P}\right)^2 = \frac{m}{P}\cdot Q^2,\quad\hbox{or}\quad
\widetilde{k}^2=\widetilde{m}\cdot Q^2,
\end{equation}
with $\widetilde{k}:=k/P$ and $\widetilde{m}:=m/P$. Furthermore, the second
identity in \eqref{2.5} implies that $\widetilde{m}$ is also a square, and
thus $\widetilde{k}$ is a multiple of $Q$. Hence
\begin{equation}\label{2.6}
k\in\left\{PQ, 2PQ,\ldots,\tfrac{Q-1}{2}\cdot PQ\right\},
\end{equation}
where we claim that the last term in \eqref{2.6} is the largest one that 
satisfies $k\leq 2\lfloor\frac{n}{4}\rfloor-1$. To see this, we note that
\[
\left\lfloor\frac{n}{4}\right\rfloor=\frac{PQ^2-r}{4},\qquad\hbox{and thus}
\qquad 2\left\lfloor\frac{n}{4}\right\rfloor-1=\frac{PQ^2-r-2}{2},
\]
where $1\leq r\leq 3$. Now we have
\[
\tfrac{Q-1}{2}\cdot PQ \leq\tfrac{1}{2}\left(PQ^2-r-2\right)\qquad
\Leftrightarrow\qquad PQ \geq r+2,
\]
and it is easy to check that the right-hand inequality holds for all $n\geq 5$.
Finally, it is obvious that $\frac{1}{2}(Q+1)PQ>2\lfloor\frac{n}{4}\rfloor-1$,
and by definition we have $A(n)=0$ for $n=1, 2, 3$; this completes the proof
of part (a).

(b) This time we rewrite $n\mid k^2$ as $k^2=4mPQ^2$, with some integer 
$m\geq 1$, and we proceed as before. In analogy to \eqref{2.5} we get
\begin{equation}\label{2.7}
\left(\frac{k}{2P}\right)^2 = \frac{m}{P}\cdot Q^2,\quad\hbox{or}\quad
\widetilde{k}^2=\widetilde{m}\cdot Q^2,
\end{equation}
with $\widetilde{k}:=k/(2P)$ and $\widetilde{m}:=m/P$. The second identity in
\eqref{2.7} means that $\widetilde{m}$ is also a square, and thus 
$\widetilde{k}$ is a multiple of $Q$. So, in analogy to \eqref{2.6} we have
\begin{equation}\label{2.8}
k\in\left\{2PQ, 4PQ,\ldots, 2(Q-1)PQ\right\},
\end{equation}
where we claim that the last term in \eqref{2.8} is the largest one that
satisfies $k\leq 2\lfloor\frac{n}{4}\rfloor-1$. This time we note that
$2\lfloor\frac{n}{4}\rfloor-1=2PQ^2-1$, and clearly
\[
2(Q-1)PQ < 2PQ^2-1 < 2Q\cdot PQ,
\]
which proves the claim. Therefore, by \eqref{2.8}, $A(n)=Q-1$, as desired.
\end{proof}

\begin{proof}[Proof of Proposition~\ref{prop:2.1}]
We begin by dealing jointly with both cases. Adapting the main ideas in 
\cite{Sh}, we fix an integer $k$ and define $N_n(k)$ to be the cardinality
\begin{equation}\label{2.9}
N_n(k):=\left|\left\{j\in{\mathbb N}:1\leq j\leq\lfloor n/4\rfloor,
\lfloor\sqrt{jn}\rfloor=k\right\}\right|.
\end{equation}
With the goal of summing over all relevant $k$, we note
that $j\leq\nu$ is equivalent to $k^2\leq n\nu$. Hence with \eqref{2.1a},
\begin{equation}\label{2.9a}
k^2\leq n\nu = (4\nu+r)\nu <4\nu^2+4\nu < (2\nu+1)^2,
\end{equation}
so that $k\leq 2\nu$. Therefore \eqref{2.1} and \eqref{2.9} imply
\begin{equation}\label{2.10}
F(n) = \sum_{k=1}^{2\nu} k\cdot N_n(k).
\end{equation}
First we evaluate $N_n(2\nu)$. By \eqref{2.9a} we have 
$\lfloor\sqrt{\nu n}\rfloor=2\nu$, while
\[
(\nu-1)n = (\nu-1)(4\nu+r) = 4\nu^2-(4-r)\nu-r < 4\nu^2
\]
since $0\leq r\leq 3$, and thus $\lfloor\sqrt{(\nu-1)n}\rfloor=2\nu-1$.
Therefore
\begin{equation}\label{2.11}
N_n(2\nu) = 1.
\end{equation}
We can now restrict our attention to $k\leq 2\nu-1$ and note that
$\left\lfloor{\sqrt{jn}}\right\rfloor=k$ if and only if
\begin{equation}\label{2.12}
k^2\leq jn < (k+1)^2\quad\Leftrightarrow\quad
\frac{k^2}{n}\leq j<\frac{(k+1)^2}{n}.
\end{equation}
From the right-hand relation in \eqref{2.12} we see that
\begin{equation}\label{2.13}
N_n(k)=\left\lfloor{\frac{(k+1)^2}{n}}\right\rfloor
-\left\lfloor{\frac{k^2}{n}}\right\rfloor+\delta_n(k),
\end{equation}
where
\[
\delta_n(k)=\begin{cases}
-1 &\hbox{when}\; n\mid(k+1)^2,\\
0 &\hbox{when}\; n\nmid k^2\;\hbox{and}\; n\nmid(k+1)^2,\\
1 &\hbox{when}\; n\mid k^2.
\end{cases}
\]
This, together with \eqref{2.10} and \eqref{2.11}, gives
\begin{equation}\label{2.14}
F(n)=2\nu+\sum_{k=1}^{2\nu-1}k\left(\left\lfloor{\frac{(k+1)^2}{n}}\right\rfloor
-\left\lfloor{\frac{k^2}{n}}\right\rfloor\right)
+\sum_{k=1}^{2\nu-1}k\cdot\delta_n(k).
\end{equation}
Let $S_1(n)$ and $S_2(n)$ be the first, resp.\ the second, sum on the right of
\eqref{2.14}. Then we have
\begin{align}
S_1(n)&=\sum_{k=1}^{2\nu-1}\left(k\left\lfloor{\frac{(k+1)^2}{n}}\right\rfloor
-(k-1)\left\lfloor{\frac{k^2}{n}}\right\rfloor\
-\left\lfloor{\frac{k^2}{n}}\right\rfloor\right)\label{2.15}\\
&=(2\nu-1)\left\lfloor{\frac{(2\nu)^2}{n}}\right\rfloor
-\sum_{k=1}^{2\nu-1}\left\lfloor{\frac{k^2}{n}}\right\rfloor\nonumber\\
&=2\nu\left\lfloor{\frac{(2\nu)^2}{n}}\right\rfloor
-\sum_{k=1}^{2\nu}\left\lfloor{\frac{k^2}{n}}\right\rfloor,\nonumber
\end{align}
where we have used telescoping. Next we rewrite
\begin{equation}\label{2.16}
S_2(n) = \sum_{\substack{k=1\\n\mid k^2}}^{2\nu-1}k
-\sum_{\substack{k=1\\n\mid(k+1)^2}}^{2\nu-1}k.
\end{equation} 
To simplify this expression, we first note that whenever $k$ occurs in the
first sum, then $k-1$ occurs in the second sum. Conversely, the final term
$k=2\nu-1$ in the second sum in \eqref{2.16} may not correspond to a term in
the first sum. This happens exactly when $n\mid(2\nu)^2$, and since 
$n=4\nu+r$, $0\leq r\leq 3$, this is the case if and only if $r=0$. Hence with
\eqref{2.16} we have 
\begin{align}
S_2(n) &= \sum_{\substack{k=1\\n\mid k^2}}^{2\nu-1}1 - \begin{cases}
2\nu-1&\hbox{when}\;r=0,\\
0&\hbox{when}\;1\leq r\leq 3,
\end{cases}\label{2.17}\\
&= A(n)  - \begin{cases}
2\nu-1&\hbox{when}\;r=0,\\
0&\hbox{when}\;1\leq r\leq 3,
\end{cases}\nonumber
\end{align}
where we have used \eqref{2.4}.

To put everything together, we begin with $r=0$, noting that
$\lfloor(2\nu)^2/n\rfloor=\nu$. Then \eqref{2.15}, \eqref{2.17}, and 
Lemma~\ref{lem:2.2}(b), combined with \eqref{2.14}, give
\begin{equation}\label{2.18}
F(4\nu)=2\nu+2\nu^2-\sum_{k=1}^{2\nu}\left\lfloor{\frac{k^2}{n}}\right\rfloor
+\overline{Q}_n-1-(2\nu-1).
\end{equation}
When $1\leq r\leq 3$, then it is easy to see with \eqref{2.1a} that 
$\lfloor(2\nu)^2/n\rfloor=\nu-1$. In this case, \eqref{2.15}, \eqref{2.17},
and Lemma~\ref{lem:2.2}(a), combined with \eqref{2.14}, give
\begin{equation}\label{2.19}
F(n)=2\nu+2\nu(\nu-1)-\sum_{k=1}^{2\nu}\left\lfloor{\frac{k^2}{n}}\right\rfloor
+\frac{Q_n-1}{2}.
\end{equation}

By division with remainder we have
\begin{equation}\label{2.20}
k^2 = \left\lfloor{\frac{k^2}{n}}\right\rfloor\cdot n + {\rm Rem}(k^2\div n),
\end{equation}
and thus
\begin{equation}\label{2.21}
\sum_{k=1}^{2\nu}\left\lfloor{\frac{k^2}{n}}\right\rfloor
=\sum_{k=1}^{2\nu}\frac{k^2}{n}
-\sum_{k=1}^{2\nu}\frac{{\rm Rem}(k^2\div n)}{n}.
\end{equation}
Using the well-known identity for sums of consecutive squares, we get
\begin{equation}\label{2.22}
\sum_{k=1}^{2\nu}\frac{k^2}{n}
=\frac{1}{6n}\cdot 2\nu(2\nu+1)(4\nu+1) = \nu\cdot\frac{8\nu^2+6\nu+1}{3n}.
\end{equation}
Substituting this last identity into \eqref{2.18} and \eqref{2.19}, we obtain
after some straightforward manipulations the desired identities \eqref{2.3} 
and \eqref{2.2}, respectively.
\end{proof}

\section{First consequences of Proposition~\ref{prop:2.1}}\label{sec:3}

As a special case of Proposition~\ref{prop:2.1}(a) we consider $n=p$, where $p$
is an odd prime. We first recall a few basic facts about quadratic residues. 
For a prime $p>2$ and an integer $a$ with 
$p\nmid a$, the number $a$ is said to be a {\it quadratic residue} modulo $p$ 
if there is an integer $k$ such that $k^2\equiv a\pmod{p}$; in this case we
write $a\in QR(p)$. When $p\nmid a$, the {\it Legendre symbol} is defined by
\begin{equation}\label{2.23}
\left(\frac{a}{p}\right) = \begin{cases}
1 &\hbox{when}\;\; a\in QR(p),\\
-1 &\hbox{when}\;\; a\not\in QR(p);
\end{cases}
\end{equation}
this is supplemented by $(\frac{a}{p})=0$ when $p|a$.
We can now prove the following connection between \eqref{2.2} and the Legendre
symbol.

\begin{lemma}\label{lem:2.3}
For any prime $p\geq 3$ we have
\begin{equation}\label{2.24}
\sum_{k=1}^{\frac{p-1}{2}}{\rm Rem}(k^2\div p) = \frac{p(p-1)}{4}
+\frac{1}{2}\sum_{j=1}^{p-1}j\left(\frac{j}{p}\right).
\end{equation}
\end{lemma}

\begin{proof}
By definition of the Legendre symbol we have
\begin{align*}
\sum_{j=1}^{p-1}j\left(\frac{j}{p}\right)
&=\sum_{j\in QR(p)}j - \sum_{j\not\in QR(p)}j
= 2\sum_{j\in QR(p)}j - \sum_{j=1}^{p-1}j \\
&= 2\sum_{k=1}^{\frac{p-1}{2}}{\rm Rem}(k^2\div p)-\frac{p(p-1)}{2},
\end{align*}
and \eqref{2.24} follows immediately.
\end{proof}

We can now obtain a proof of Proposition~\ref{prop:1.1}. We consider  
Lemma~\ref{lem:2.3} for $p\equiv 1\pmod{4}$ and use the well-known fact that 
\begin{equation}\label{2.25}
\sum_{j=1}^{p-1}j\left(\frac{j}{p}\right)=0\qquad (p\equiv 1\pmod{4}).
\end{equation}
The identity \eqref{2.25} can be shown, for instance, by noting that in this 
case we have $(\frac{p-j}{p})=(\frac{j}{p})$, and by changing the direction 
of summation we get
\begin{equation}\label{2.26}
\sum_{j=1}^{p-1}j\left(\frac{j}{p}\right)
=\sum_{j=1}^{p-1}(p-j)\left(\frac{p-j}{p}\right)
=p\sum_{j=1}^{p-1}\left(\frac{j}{p}\right)
-\sum_{j=1}^{p-1}j\left(\frac{j}{p}\right).
\end{equation}
The first sum on the right is 0 since there are as many quadratic residues as
nonresidues, and then \eqref{2.26} implies \eqref{2.25}.

Now \eqref{2.25}, together with \eqref{2.24} and \eqref{2.2} with $Q_p=1$,
leads to $F(p)=(p^2-1)/12$. This is, basically, the idea behind
Shirali's proof in \cite{Sh} of Proposition~\ref{prop:1.1}.

Since \eqref{2.25} does not hold for $p\equiv 3\pmod{4}$, Lemma~\ref{lem:2.3}
will be different in this case; we will deal with this in the next section.

\medskip
Quite surprisingly, Proposition~\ref{prop:1.1} remains true when we replace the
prime $p\equiv 1\pmod{4}$ with a product of distinct primes of this form. We 
will see that this, and a related identity, are consequences of a result by 
Shirali. We begin with a lemma.

\begin{lemma}\label{lem:4.1}
For any odd integer $n\geq 3$ we have
\begin{equation}\label{4.1}
\sum_{k=1}^{n-1}{\rm Rem}(k^2\div 2n) = \frac{n(n-1)}{2}
+2\sum_{k=1}^{\frac{n-1}{2}}{\rm Rem}(k^2\div n).
\end{equation}
\end{lemma}

\begin{proof}
We claim that of the smallest nonnegative remainders of $k^2$ and $(n-k)^2$
modulo $2n$, one is $<n$, while the other is $>n$.
Indeed, suppose that some $k$ with $1\leq k\leq n-1$ is such that
$k^2=2ns+a$, $1\leq a\leq n-1$, for some integer $s$. Then
\[
(n-k)^2\equiv n^2+k^2\equiv n^2+a = 2n\cdot\tfrac{n-1}{2}+n+a
\equiv n+a\pmod{2n},
\]
which proves the claim. So there are as many remainders modulo $2n$ that are
$>n$ as are $<n$, namely $\frac{n-1}{2}$, and the excess in the sum is
$\frac{n-1}{2}\cdot n$, while obviously $(n-k)^2\equiv k^2\pmod{n}$. This proves
the identity \eqref{4.1}.
\end{proof}

We now state Shirali's result as another lemma.

\begin{lemma}[Shirali \cite{Sh}]\label{lem:5.2}
Let $n$ be a product of primes that are congruent to $1\pmod{4}$, not 
necessarily distinct, and write $n=P_n\cdot Q_n^2$, with $P_n$ squarefree. Then
\begin{equation}\label{5.2}
\sum_{k=1}^{\frac{n-1}{2}}{\rm Rem}(k^2\div P) = \frac{n(n-Q_n)}{4}.
\end{equation}
\end{lemma}

\begin{proposition}\label{prop:5.1}
With $n$ as in Lemma~\ref{lem:5.2}, we have
\begin{equation}\label{5.1}
f(n)=\frac{1}{4}\left(Q_n-1\right)\qquad\hbox{and}\qquad 
f(2n)=\frac{1}{4}\left(Q_n-1-n\right).
\end{equation}
In particular, if $n$ is squarefree, then
\begin{equation}\label{5.1a}
f(n)=0\qquad\hbox{and}\qquad f(2n)=-\frac{n}{4}.
\end{equation}
\end{proposition}

\begin{proof}
To obtain the first identity in \eqref{5.1}, we use \eqref{2.2} and note that
$\nu=(n-1)/4$. Then with \eqref{1.1} we get after some easy manipulations,
\[
f(n)=\frac{1}{2}Q_n-\frac{n+1}{4}
+\frac{1}{n}\sum_{k=1}^{\frac{n-1}{2}}{\rm Rem}(k^2\div n),
\]
and with \eqref{5.2} we get the first identity in \eqref{5.1}.

Next we replace $n$ by $2n$ in \eqref{2.2} and note that $\nu=(n-1)/2$. Then
with \eqref{1.1} we get
\begin{equation}\label{5.3}
f(2n)=\frac{1}{2}Q_n-\frac{3n}{4} 
+\frac{1}{2n}\sum_{k=1}^{n-1}{\rm Rem}(k^2\div 2n).
\end{equation}
Now, with \eqref{4.1} and \eqref{5.2} we have
\[
\sum_{k=1}^{n-1}{\rm Rem}(k^2\div 2n) = \frac{n(n-1)}{2}
+2\sum_{k=1}^{\frac{n-1}{2}}{\rm Rem}(k^2\div n) 
= n\left(n-\frac{Q_n+1}{2}\right);
\]
this, combined with \eqref{5.3}, gives the second identity in \eqref{5.1}.
Finally, when $n$ is squarefree then $Q_n=1$, and \eqref{5.1a} follows 
immediately from \eqref{5.1}.
\end{proof}

\section{Class numbers of imaginary quadratic fields}\label{sec:4}

To set the stage, we recall a few basic facts from elementary and algebraic
number theory; for further details we refer the reader to any introduction  to
algebraic number theory, for instance \cite{AW}. A {\it quadratic field} is a
field extension of the rationals of the form ${\mathbb Q}(\sqrt{n})$, where
$n\neq 0$ is a squarefree integer. When $n>0$, the field is called {\it real\/},
otherwise {\it imaginary\/}. The {\it discriminant} $d$ of 
${\mathbb Q}(\sqrt{n})$ is given by
\begin{equation}\label{3.1}
d = \begin{cases}
n &\hbox{when}\;\;n\equiv 1\pmod{4},\\
4n &\hbox{when}\;\;n\not\equiv 1\pmod{4}.
\end{cases}
\end{equation}
When $d<0$, the number of units $w(d)$ in the ring of integers of 
${\mathbb Q}(\sqrt{n})$ is 2 when $d<-4$, while $w(-3)=6$ and $w(-4)=4$.

Next we recall two extensions of the Legendre symbol. First,
the {\it Jacobi symbol} extends \eqref{2.23} as follows: if 
$m=p_1^{\alpha_1}\cdots p_r^{\alpha_r}$, where $p_1,\ldots,p_r$ are odd primes,
then 
\[
\left(\frac{a}{m}\right) = \prod_{j=1}^r\left(\frac{a}{p_j}\right)^{\alpha_j}.
\]
The {\it Kronecker symbol} extends this further by
\[
\left(\frac{a}{2}\right) = \begin{cases}
\left(\frac{2}{a}\right) &a\;\;\hbox{odd},\\
0&a\;\;\hbox{even};
\end{cases}\qquad
\left(\frac{a}{-1}\right) = \begin{cases}
1 &a>0,\\
-1 &a<0;
\end{cases}\qquad
\]
see, e.g., \cite[p.~36]{Co1}. Now we can state Dirichlet's class number 
formula for imaginary quadratic number fields, as given in \cite[p.~322]{AW}.
See also \cite[p.~342ff.]{BS}, where a proof can be found.

\begin{theorem}[Dirichlet, 1839]\label{thm:3.1}
Let $K$ be a quadratic number field with discriminant $d<0$. Then the class
number of $K$ is
\begin{equation}\label{3.2}
h(K)=\frac{-w(d)}{2|d|}\sum_{j=1}^{|d|-1}j\left(\frac{d}{j}\right),
\end{equation}
where $(\frac{d}{j})$ is the Kronecker symbol.
\end{theorem}

In this paper we are mainly interested in quadratic fields of the form
${\mathbb Q}(\sqrt{-p})$, where $p$ is an odd prime. Due to the discriminant
identity \eqref{3.1}, it is convenient to distinguish between two cases. 
In what follows we use the notation $h(-p)=h(K)$ when 
$K={\mathbb Q}(\sqrt{-p})$.
It is quite likely that the identities in Corollary~\ref{cor:3.2} 
can be found in the literature.

\begin{corollary}\label{cor:3.2}
For any prime $p\equiv 1\pmod{4}$ we have
\begin{equation}\label{3.3}
h(-p)=\frac{1}{2}\sum_{\substack{j=1\\j\,odd}}^{2p-1}
(-1)^{\frac{j-1}{2}}\left(\frac{j}{p}\right),
\end{equation}
and equivalently
\begin{equation}\label{3.4}
h(-p)=\frac{1}{2p}\sum_{\substack{j=1\\j\,odd}}^{2p-1}
(-1)^{\frac{j-1}{2}}j\left(\frac{j}{p}\right),
\end{equation}
where $(\frac{j}{p})$ is the Legendre symbol defined by \eqref{2.23}.
\end{corollary}

\begin{proof}
When $p\equiv 1\pmod{4}$, then by \eqref{3.1} we have $d=-4p$. Using 
properties of the Kronecker and Jacobi symbols, not all listed above (such as
quadratic reciprocity), we have $(\frac{-4p}{j})=0$ when $j$ is even, while
\[
\left(\frac{-4p}{j}\right)=\left(\frac{-4}{j}\right)\left(\frac{p}{j}\right)
=\left(\frac{-1}{j}\right)\left(\frac{p}{j}\right)
=(-1)^{\frac{j-1}{2}}\left(\frac{j}{p}\right)\qquad (j\;\hbox{odd}).
\]
Since $w(d)=2$, \eqref{3.2} now gives
\begin{align}
h(-p)&=\frac{-1}{4p}\sum_{\substack{j=1\\j\,odd}}^{4p-1}
(-1)^{\frac{j-1}{2}}j\left(\frac{j}{p}\right)\label{3.5} \\
&=\frac{-1}{4p}\sum_{\substack{j=1\\j\,odd}}^{2p-1}
\left((-1)^{\frac{j-1}{2}}j\left(\frac{j}{p}\right)
+(-1)^{\frac{2p+j-1}{2}}(2p+j)\left(\frac{2p+j}{p}\right)\right).\nonumber
\end{align}
We note that
\[
(-1)^{\frac{2p+j-1}{2}} = - (-1)^{\frac{j-1}{2}}\quad\hbox{and}\quad
\left(\frac{2p+j}{p}\right) = \left(\frac{j}{p}\right),
\]
so \eqref{3.5} simplifies to
\[
h(-p)=\frac{-1}{4p}\sum_{\substack{j=1\\j\,odd}}^{2p-1}
(-1)^{\frac{j-1}{2}}(-2p)\left(\frac{j}{p}\right),
\]
which immediately gives \eqref{3.3}.

To obtain \eqref{3.4}, we sum the last term in \eqref{3.5} differently, namely
\begin{align}
\sum_{\substack{j=1\\j\,odd}}^{2p-1}
&(-1)^{\frac{2p+j-1}{2}}(2p+j)\left(\frac{2p+j}{p}\right)\label{3.6}\\
&=-2\sum_{\substack{j=1\\j\,odd}}^{2p-1}
(-1)^{\frac{j-1}{2}}j\left(\frac{j}{p}\right)
-\sum_{\substack{j=1\\j\,odd}}^{2p-1}
(-1)^{\frac{j-1}{2}}(2p-j)\left(\frac{j}{p}\right) \nonumber\\
&=-2\sum_{\substack{j=1\\j\,odd}}^{2p-1}
(-1)^{\frac{j-1}{2}}j\left(\frac{j}{p}\right)
-\sum_{\substack{j=1\\j\,odd}}^{2p-1}
(-1)^{\frac{2p-j-1}{2}}j\left(\frac{2p-j}{p}\right),\nonumber
\end{align}
where we have reversed the order of summation in the last sum. Now
\[
\left(\frac{2p-j}{p}\right) = \left(\frac{-j}{p}\right)
= \left(\frac{-1}{p}\right)\left(\frac{j}{p}\right)
= \left(\frac{j}{p}\right),
\]
by the first complementary law of quadratic reciprocity, since 
$p\equiv 1\pmod{4}$. Furthermore,
\[
(-1)^{\frac{2p-j-1}{2}} = (-1)^{\frac{j-1}{2}}\quad\hbox{since}\quad
\frac{2p-j-1}{2}-\frac{j-1}{2}=p-j\equiv 0\pmod{2}
\]
because both $p$ and $j$ are odd. Hence with \eqref{3.6},
\[
\sum_{\substack{j=1\\j\,odd}}^{2p-1}
(-1)^{\frac{2p+j-1}{2}}(2p+j)\left(\frac{2p+j}{p}\right)
=-3\sum_{\substack{j=1\\j\,odd}}^{2p-1}
(-1)^{\frac{j-1}{2}}j\left(\frac{j}{p}\right),
\]
and combining this with \eqref{3.5}, we get \eqref{3.4}.
\end{proof}

The next corollary is well known and is attributed to Jacobi; see, e.g.,
\cite[Ch.~6]{Da}. For the sake of completeness we show how it follows from 
Dirichlet's formula.

\begin{corollary}\label{cor:3.3}
For any prime $p\equiv 3\pmod{4}$, $p\neq 3$, we have
\begin{equation}\label{3.7}
h(-p)=-\frac{1}{p}\sum_{j=1}^{p-1}j\left(\frac{j}{p}\right),
\end{equation}
where $(\frac{j}{p})$ is the Legendre symbol defined by \eqref{2.23}.
\end{corollary}

\begin{proof}
When $p\equiv 3\pmod{4}$, then by \eqref{3.1} we have $d=-p$, and when
$p\neq 3$, then $w(d)=2$. Furthermore, using Proposition~2.2.6 in
\cite[p.~36]{Co1}, we have the Kronecker symbol identities
\[
\left(\frac{d}{j}\right)=\left(\frac{-p}{j}\right)=\left(\frac{j}{-p}\right)
=\left(\frac{j}{-1}\right)\left(\frac{j}{p}\right)=\left(\frac{j}{p}\right).
\]
This, together with \eqref{3.2}, yields \eqref{3.7}.
\end{proof}

As a first application of Corollary~\ref{cor:3.3} we derive the identity
\eqref{1.6}.

\begin{proof}[Proof of Proposition~\ref{prop:1.2}]
We use \eqref{2.2} with $n=p\equiv 3\pmod{4}$, so that $\nu=\frac{p-3}{4}$. Then
\begin{align}
\sum_{k=1}^{2\nu}{\rm Rem}(k^2\div n)
&=\sum_{k=1}^{\frac{p-1}{2}}{\rm Rem}(k^2\div p)
-{\rm Rem}((\tfrac{p-1}{2})^2\div p)\label{3.8}\\
&=\frac{p(p-1)}{4}+\frac{1}{2}\sum_{j=1}^{p-1}j\left(\frac{j}{p}\right)
-\frac{p+1}{4},\nonumber
\end{align}
where we have used Lemma~\ref{lem:2.3} and the fact that
\begin{equation}\label{3.9}
\left(\frac{p-1}{2}\right)^2-\frac{p+1}{4}=p\cdot\frac{p-3}{4}\equiv 0\pmod{p}.
\end{equation}
Hence with the class number formula \eqref{3.7}, the identity \eqref{3.8}
becomes
\[
\sum_{k=1}^{2\nu}{\rm Rem}(k^2\div n)
=\frac{p^2-2p-1}{4}-\frac{p}{2}\cdot h(-p).
\]
Finally, we substitute this into \eqref{2.2} and recall that 
$f(p)=F(p)-(p^2-1)/12$; then we get the desired identity \eqref{1.6} after some
straightforward manipulations, using again the fact that $Q_p=1$.
\end{proof}

\section{Evaluating $f(2p)$ and $f(4p)$}\label{sec:5}

Before proceeding to more general arguments of the function $f(n)$, we consider
$n=2p$ and $n=4p$, where $p$ is an odd prime. The case $n=2p$ turns out to be
quite straightforward and is based on Lemma~\ref{lem:4.1}.

\begin{proposition}\label{prop:4.2}
For any prime $p\geq 3$ we have
\begin{equation}\label{4.2}
f(2p)= -\frac{p}{4} + \frac{1}{2p}\sum_{j=1}^{p-1}j\left(\frac{j}{p}\right),
\end{equation}
and in particular, for $p\geq 5$, 
\begin{equation}\label{4.3}
f(2p)=\begin{cases}
-\frac{p}{4} &\hbox{when}\;\;p\equiv 1\pmod{4},\\
-\frac{p}{4}-\frac{1}{2}h(-p) &\hbox{when}\;\;p\equiv 3\pmod{4}.
\end{cases}
\end{equation}
\end{proposition}

\begin{proof}
We use \eqref{2.2} with $n=2p$ and note that $Q_{2p}=1$. Then 
$\nu=\frac{p-1}{2}$, and with \eqref{4.1} for $n=p$ we get after some easy 
manipulations,
\[
f(2p)=\frac{1}{4}-\frac{p}{2}
+\frac{1}{p}\sum_{k=1}^{\frac{p-1}{2}}{\rm Rem}(k^2\div p).
\]
This, with Lemma~\ref{lem:2.3}, gives \eqref{4.2}.

When $p\equiv 1\pmod{4}$, we already saw in \eqref{2.25} that the sum on the 
right of \eqref{4.2} vanishes, which leads to the first part of \eqref{4.3}. For
$p\equiv 3\pmod{4}$, the class number formula \eqref{3.7} immediately gives the
second part of \eqref{4.3}.
\end{proof}

The first part of \eqref{4.3} is actually a special case of \eqref{5.1a}.
We now turn to the more challenging determination of $f(4p)$.

\begin{proposition}\label{prop:4.3}
For any prime $p\geq 5$ we have
\begin{equation}\label{4.4}
f(4p)= \frac{1}{4} + \frac{p}{2} -\delta(p)h(-p),
\end{equation}
where 
\[
\delta(p)=\begin{cases}
1/2, & p\equiv 1\pmod{4},\\
2, & p\equiv 3\pmod{8},\\
1, & p\equiv 7\pmod{8}.
\end{cases}
\]
\end{proposition}

We prove this through a sequence of lemmas which may be of interest in their
own rights.

\begin{lemma}\label{lem:4.4}
For any prime $p\equiv 1\pmod{4}$, let
\begin{equation}\label{4.5}
S_1(p):=\sum_{j\in A_1(p)}j,\qquad S_3(p):=\sum_{j\in A_3(p)}j,
\end{equation}
where $A_r(p):=\{j\in{\mathbb N}\mid j\leq 2p-1, j\equiv r\pmod{4},
(\frac{j}{p})=1\}$. Then
\begin{equation}\label{4.6}
h(-p)=\frac{1}{p}\left(S_1(p)-S_3(p)\right).
\end{equation}
\end{lemma}

\begin{proof}
Using the definition of the Legendre symbol, we can rewrite
\[
S_1(p)=\sum_{\substack{j=1\\j\equiv 1 (4)}}^{2p-1}
\frac{j}{2}\left(\left(\frac{j}{p}\right)+1\right)-\frac{p}{2},\qquad
S_3(p)=\sum_{\substack{j=3\\j\equiv 3 (4)}}^{2p-1}
\frac{j}{2}\left(\left(\frac{j}{p}\right)+1\right),
\]
where in the case of $S_1(p)$ we needed to subtract $p/2$ since 
$(\frac{j}{p})=0$ for $j=p$. These two identities then combine to give
\begin{equation}\label{4.7}
\frac{S_1(p)-S_3(p)}{p}=\frac{1}{2p}\sum_{\substack{j=1\\j\,odd}}^{2p-1}
(-1)^{\frac{j-1}{2}}j\left(\frac{j}{p}\right)-\frac{1}{2}
+\frac{1}{2p}\sum_{\substack{j=1\\j\,odd}}^{2p-1}(-1)^{\frac{j-1}{2}}j.
\end{equation}
The first sum on the right is the right-hand side of \eqref{3.4}, while the
second sum can be evaluated by splitting it into two:
\[
\frac{1}{2}\sum_{\substack{j=1\\j\,odd}}^{2p-1}(-1)^{\frac{j-1}{2}}j
+\frac{1}{2}\sum_{\substack{j=1\\j\,odd}}^{2p-1}(-1)^{\frac{2p-j+1}{2}}(2p-j)
=\frac{2p}{2}\sum_{\substack{j=1\\j\,odd}}^{2p-1}(-1)^{\frac{j-1}{2}} = p.
\]
This and \eqref{3.4} substituted into \eqref{4.7} then gives \eqref{4.6}.
\end{proof}

\begin{lemma}\label{lem:4.5}
Let $p\equiv 1\pmod{4}$ be a prime. Then
\begin{equation}\label{4.8}
h(-p)=\frac{1}{p}\sum_{\substack{k=1\\k\,odd}}^{p-1}
\left(2p-{\rm Rem}(k^2\div 4p)\right).
\end{equation}
\end{lemma}

\begin{proof}
We will show that the summands on the right of \eqref{4.8} are exactly those of
$S_1(p)$ and $S_3(p)$ in the previous lemma. For greater ease of notation we
set $a_k:={\rm Rem}(k^2\div 4p)$. By definition, $k^2\equiv a_k\pmod{4p}$, which
implies $k^2\equiv a_k\pmod{p}$, and so $a_k\in QR(p)$. If we set $r_k:=2p-a_k$,
then 
\[
\left(\frac{r_k}{p}\right)=\left(\frac{2p-a_k}{p}\right)
=\left(\frac{-1}{p}\right)\left(\frac{a_k}{p}\right) = 1,
\]
since $p\equiv 1\pmod{4}$; therefore $r_k\in QR(p)$ as well. Next, by 
definition, we have $a_k=k^2+4pm\equiv 1\pmod{4}$ for odd $k$ and some 
$m\in{\mathbb Z}$. Since $p\equiv 1\pmod{4}$, this implies 
$r_k=2p-a_k\equiv 2-1=1\pmod{4}$.

Furthermore, we claim that the summands $r_k$ are distinct for 
$k=1, 3, 5,\ldots, p-2$. Indeed, if $r_j\equiv r_k\pmod{4p}$, then
\[
j^2\equiv k^2\pmod{4p}\quad\Leftrightarrow\quad 4p\mid(j-k)(j+k).
\]
Now $j+k$ is even, $j+k\in\{2, 4, 6,\ldots, 2p-4\}$, and thus $p\nmid j+k$.
Hence $p\mid j-k$. But $1\leq j, k\leq p-2$, and therefore $j=k$, as claimed.

We have thus shown that the summands $r_k$ consist of $\frac{p-1}{2}$ odd
integers between $-2p+1$ and $2p-1$, all are congruent to 1 (mod 4), and they
are all quadratic residues modulo $p$. But these are exactly the elements
$j\in S_1(p)$ and $-k$, where $k\in S_3(p)$. Hence \eqref{4.6} implies 
\eqref{4.8}.
\end{proof}

In the next lemma we remove the restriction ``$k$ odd" from the summation.

\begin{lemma}\label{lem:4.6}
Let $p\equiv 1\pmod{4}$ be a prime. Then
\begin{equation}\label{4.9}
h(-p)=\frac{1}{p}\sum_{k=1}^{p-1}\left(2p-{\rm Rem}(k^2\div 4p)\right).
\end{equation}
\end{lemma}

\begin{proof}
Comparing \eqref{4.9} with \eqref{4.8}, it remains to be shown that the sum over
all even $k$ in \eqref{4.9} vanishes. Setting $k=2j$ in this sum and dividing
both sides by 4, we see that this is equivalent to
\[
\sum_{j=1}^{\frac{p-1}{2}}\left(\frac{p}{2}-{\rm Rem}(j^2\div p)\right)=0,
\quad\hbox{or}\quad 
\sum_{j=1}^{\frac{p-1}{2}}{\rm Rem}(j^2\div p)=\frac{p(p-1)}{4}.
\]
But this last identity follows from \eqref{2.24} and \eqref{2.25}.
\end{proof}

\begin{lemma}\label{lem:4.7}
Let $p\equiv 3\pmod{4}$ be a prime, $p\neq 3$. Then
\begin{equation}\label{4.10}
\frac{1}{2p}\sum_{k=1}^{p-1}{\rm Rem}(k^2\div 4p)
= p-\frac{3}{2}+\big(\varepsilon(p)-2\big)h(-p),
\end{equation}
where $\varepsilon(p)=0$ when $p\equiv 3\pmod{8}$ and $\varepsilon(p)=1$ 
when $p\equiv 7\pmod{8}$.
\end{lemma}

\begin{proof}
Combining \eqref{2.24} with \eqref{3.7} and using symmetry on the left of 
\eqref{2.24}, we get
\begin{equation}\label{4.11}
\sum_{k=1}^{p-1}{\rm Rem}(k^2\div p) = \frac{p(p-1)}{2}
-p\cdot h(-p)\qquad (p\equiv 3\pmod{4}),
\end{equation}
where $p>3$ is a prime. Hence we are done if we can evaluate the expression 
\begin{equation}\label{4.12}
S(p):=\sum_{k=1}^{p-1}{\rm Rem}(k^2\div 4p)
-4\sum_{k=1}^{p-1}{\rm Rem}(k^2\div p).
\end{equation}
To do so, we first note that Rem$((2k)^2\div 4p)=4\cdot{\rm Rem}(k^2\div p)$, 
so that
\begin{align}
S(p)&=\sum_{k=1}^{\frac{p-1}{2}}{\rm Rem}((2k-1)^2\div 4p)
-4\sum_{k=\frac{p+1}{2}}^{p-1}{\rm Rem}(k^2\div p)\label{4.13}\\
&=\sum_{k=1}^{\frac{p-1}{2}}{\rm Rem}((p-2k)^2\div 4p)
-4\sum_{k=1}^{\frac{p-1}{2}}{\rm Rem}(k^2\div p),\nonumber
\end{align}
where in both sums we have reversed the order of summation.
Next, since $p\equiv 3\pmod{4}$, we have $p^2\equiv 3p\pmod{4p}$, and thus
\begin{equation}\label{4.14}
(p-2k)^2 = p^2-4pk+4k^2 \equiv 3p+4k^2\pmod{4p}.
\end{equation}
When Rem$(4k^2\div 4p)<p$, which is equivalent to Rem$(k^2\div p)<\frac{p}{4}$,
then \eqref{4.14} gives
\begin{equation}\label{4.15}
{\rm Rem}((p-2k)^2\div 4p) - 4\cdot {\rm Rem}(k^2\div p) = 3p.
\end{equation}
On the other hand, when $p<{\rm Rem}(4k^2\div 4p)<4p$, then we have
$(p-2k)^2 = {\rm Rem}(4k^2\div 4p)-p$, and thus
\begin{equation}\label{4.16}
{\rm Rem}((p-2k)^2\div 4p) - 4\cdot {\rm Rem}(k^2\div p) = -p.
\end{equation} 
Now, the number of $k$, $1\leq k\leq\frac{p-1}{2}$, for which
Rem$(k^2\div p)<\frac{p}{4}$, is exactly the number $N$ of quadratic residues
between 1 and $p/4$, that is, 
\[
N=\sum_{\substack{j=1\\(\frac{j}{p})=1}}^{\lfloor p/4\rfloor}1
=\frac{1}{2}\sum_{j=1}^{\frac{p-3}{4}}\left(\left(\frac{j}{p}\right)+1\right)
=\frac{1}{2}\sum_{j=1}^{\frac{p-3}{4}}\left(\frac{j}{p}\right)
+\frac{p-3}{8}.
\]
The sum of the Legendre symbols on the right is known to be
\[
\sum_{j=1}^{\frac{p-3}{4}}\left(\frac{j}{p}\right) = \begin{cases}
0,& p\equiv 3\pmod{8},\\
h(-p),& p\equiv 7\pmod{8};
\end{cases}
\]
see, for instance, the particularly well-organized tables in \cite{JM}. Hence
\begin{equation}\label{4.17}
N=\frac{p-3}{8}+\frac{1}{2}\varepsilon(p)h(-p), 
\end{equation}
with $\varepsilon(p)$ as defined after \eqref{4.10}. Using the definition of
$N$ together with \eqref{4.13}, \eqref{4.15} and \eqref{4.16}, we get
\begin{equation}\label{4.18}
S(p) = 3pN-p\left(\frac{p-1}{2}-N\right) = 4pN-\frac{p(p-1)}{2}\
= -p+2p\varepsilon(p)h(-p),
\end{equation}
where we have used \eqref{4.17}. Finally, combining \eqref{4.12} with 
\eqref{4.18} and \eqref{4.11} gives
\begin{align*}
\sum_{k=1}^{p-1}{\rm Rem}(k^2\div 4p)
&= -p+2p\varepsilon(p)h(-p)+4\left(\frac{p(p-1)}{2}-ph(-p)\right)\\
&= 2p^2-3p+2p\big(\varepsilon(p)-2)h(-p),
\end{align*}
and upon dividing everything by $2p$ we get \eqref{4.10}.
\end{proof}

\begin{proof}[Proof of Proposition~\ref{prop:4.3}]
By \eqref{1.4} and \eqref{2.3} with $\nu=p$, and noting that 
$\overline{Q}_p=1$, we have
\begin{equation}\label{4.19}
f(4p)=1-\frac{p}{2}+\frac{1}{4p}\sum_{k=1}^{2p}{\rm Rem}(k^2\div 4p).
\end{equation}
Clearly, ${\rm Rem}(k^2\div 4p)=0$ for $k=2p$. Furthermore, if 
$p\equiv r\pmod{4}$, then 
$p^2=4p\frac{p-r}{4}+rp\equiv p\pmod{4p}$, and we have Rem$(p^2\div 4p)=rp$. We
also have $(2p-k)^2\equiv k^2\pmod{4p}$, and thus we can rewrite \eqref{4.19} as
\begin{equation}\label{4.20}
f(4p)=1+\frac{r}{4}-\frac{p}{2}
+\frac{1}{2p}\sum_{k=1}^{p-1}{\rm Rem}(k^2\div 4p).
\end{equation}
Finally, we rewrite \eqref{4.9} as
\[
\frac{1}{2p}\sum_{k=1}^{p-1}{\rm Rem}(k^2\div 4p) = p-1-\frac{1}{2}h(-p),
\qquad (p\equiv 1\pmod{4}).
\]
Then this and \eqref{4.10}, substituted into \eqref{4.20}, gives the desired 
identity \eqref{4.4}.
\end{proof}

\section{$f(n)$ when $n$ is a power of a prime}\label{sec:6}

We recall that Proposition~\ref{prop:5.1} dealt with squarefree integers.
The other extreme is the case of powers of a prime; see Table~2 for some 
small values.

\bigskip
\begin{center}
\begin{tabular}{|c||r|r|r|r|r|r|r|}
\hline
$\alpha$ & $4f(2^\alpha)$ & $3f(3^\alpha)$ & $f(5^\alpha)$ & $f(7^\alpha)$ & 
$f(11^\alpha)$ & $f(13^\alpha)$ & $f(17^\alpha)$ \\
\hline
1 & $-1$ & $-2$ & 0 & $-2$ & $-3$ & 0 & 0 \\
2 &   3  &   1  & 1 &   1  &   2  & 3 & 4\\
3 &   3  & $-20$& 1 & $-88$& $-336$ & 3 & 4\\
4 &   7  &   4  & 6 &   8  &   24 & 42 & 72\\
5 &  11  & $-182$& 6& $-4\,218$ & $-40\,299$ & 42 & 72 \\
6 &  27  &  13  & 31 & 57  & 266  & 549 & $1\,228$\\
7 &  51  & $-1\,640$& 31 & $-206\,000$ & $-4\,872\,192$ & 549 & $1\,228$ \\
8 &  115 & 40 & 156 & 400 & $2\,928$ & $7\,140$ & $20\,880$\\
\hline
\end{tabular}

\medskip
{\bf Table~2}: $f(p^\alpha)$ for $1\leq\alpha\leq 8$ and $2\leq p\leq 17$.
\end{center}
\bigskip

We observe that for a fixed prime $p$, the sequence
$\{f(p^\alpha)\}_{\alpha\geq 1}$ exhibits very different behaviour, depending
on whether  $p=2$, $p\equiv 1\pmod{4}$, or $p\equiv 3\pmod{4}$. 
An explanation is provided by the following main result of this section.

\begin{proposition}\label{prop:6.1}
$(a)$ For an integer $\beta\geq 1$ we have
\begin{align}
f(2^{2\beta})&= 2^{2\beta-3}-2^{\beta-2}+\frac{3}{4},\label{6.1}\\
f(2^{2\beta+1})&= 2^{2\beta-2}-2^{\beta-1}+\frac{3}{4}.\label{6.2}
\end{align}
$(b)$ When $p\equiv 1\pmod{4}$ and $\beta\geq 1$, then
\begin{equation}\label{6.3}
f(p^{2\beta})=f(p^{2\beta+1}) = \frac{1}{4}\left(p^{\beta}-1\right).
\end{equation}
$(c)$ When $p\equiv 3\pmod{4}$ and $\beta\geq 1$, then
\begin{align}
f(p^{2\beta})&= \frac{1}{4}\left(p^{\beta}-1\right)
\left(1-\frac{2}{p-1}h^*(-p)\right),\label{6.4}\\
f(p^{2\beta+1})&= \frac{-1}{4}\left(p^{\beta+1}-1\right)
\left(p^{\beta}+\frac{2}{p-1}h^*(-p)\right),\label{6.5}
\end{align}
where $h^*(-p)=h(-p)$ when $p\geq 7$, and $h^*(-3)=1/3$.
\end{proposition}

The second part of \eqref{6.3} and the identity \eqref{6.5} also hold for
$\beta=0$; in this case we recover Propositions~\ref{prop:1.1} 
and~\ref{prop:1.2}, respectively.

In view of Proposition~\ref{prop:2.1} it is clear that the main ingredient
in the proof of Proposition~\ref{prop:6.1} would be the evaluation of the 
respective sums of remainders on the right of \eqref{2.3} or \eqref{2.4}. 
In fact, the following identities hold.

\begin{lemma}\label{lem:6.2}
$(a)$ For any integer $\beta\geq 1$ we have
\begin{align}
\sum_{k=1}^{2^{2\beta-1}}{\rm Rem}\left(k^2\div 2^{2\beta}\right)
&= 2^{2\beta-2}\left(2^{2\beta}-3\cdot 2^{\beta}+3\right),\label{6.6}\\
\sum_{k=1}^{2^{2\beta}}{\rm Rem}\left(k^2\div 2^{2\beta+1}\right)
&= 2^{2\beta-1}\left(2^{2\beta+1}-4\cdot 2^{\beta}+3\right).\label{6.7}
\end{align}
$(b)$ If $p\equiv 1\pmod{4}$ is a prime, then
\begin{align}
\sum_{k=1}^{\frac{p^{2\beta}-1}{2}}{\rm Rem}\left(k^2\div p^{2\beta}\right)
&=\frac{1}{4}p^{3\beta}\left(p^{\beta}-1\right)\qquad(\beta\geq 1),\label{6.8}\\
\sum_{k=1}^{\frac{p^{2\beta+1}-1}{2}}{\rm Rem}\left(k^2\div p^{2\beta+1}\right)
&=\frac{1}{4}p^{3\beta+1}\left(p^{\beta+1}-1\right)\qquad(\beta\geq 0).\label{6.9}
\end{align}
$(c)$ If $p\equiv 3\pmod{4}$ is a prime, then for $\beta\geq 1$, resp.\
$\beta\geq 0$,
\begin{align}
\sum_{k=1}^{\frac{p^{2\beta}-1}{2}}{\rm Rem}\left(k^2\div p^{2\beta}\right)
&=\frac{1}{4}p^{2\beta}\left(p^{\beta}-1\right)
\left(p^{\beta}-\frac{2}{p-1}h^*(-p)\right),\label{6.10}\\
\sum_{k=1}^{\frac{p^{2\beta+1}-1}{2}}{\rm Rem}\left(k^2\div p^{2\beta+1}\right)
&=\frac{1}{4}p^{2\beta+1}\left(p^{\beta+1}-1\right)
\left(p^{\beta}-\frac{2}{p-1}h^*(-p)\right).\label{6.11}
\end{align}
\end{lemma}


With this lemma it is now straightforward, though tedious, to obtain 
Proposition~\ref{prop:6.1}.

\begin{proof}[Proof of Proposition~\ref{prop:6.1}]
(a) By definition of $\overline{Q}_n$ in Corollary~\ref{cor:2.2} we have
\[
\overline{Q}_{2^{2\beta}} = \overline{Q}_{2^{2\beta+1}} = 2^{\beta-1}.
\]
We substitute this and \eqref{6.6}, resp.\ \eqref{6.7}, into \eqref{2.3a} and 
obtain \eqref{6.1}, resp.\ \eqref{6.2}.

(b) When $n=p^{2\beta}$, resp.\ $n=p^{2\beta+1}$, we have 
$\nu=(p^{2\beta}-1)/4$, resp.\ $\nu=(p^{2\beta+1}-1)/4$, and $Q_n=p^{\beta}$
in both cases. Then \eqref{2.2} and \eqref{1.4}, together with \eqref{6.8},
resp.\ \eqref{6.9}, give both parts of \eqref{6.3}.

(c) This is very similar to part (b), except that for $n=p^{2\beta+1}$ we have
$\nu=(p^{2\beta+1}-3)/4$. We therefore require the remainder
\begin{equation}\label{6.11a}
{\rm Rem}\left(\left(\frac{p^{2\beta+1}-1}{2}\right)^2\div p^{2\beta+1}\right)
=\frac{1}{4}\left(p^{2\beta+1}+1\right)\qquad (\beta\geq 0),
\end{equation}
which can be verified by considering the difference
\[
\left(\frac{p^{2\beta+1}-1}{2}\right)^2-\frac{p^{2\beta+1}+1}{4}
=p^{2\beta+1}\cdot\frac{1}{4}\left(p^{2\beta+1}-3\right)
\equiv 0\pmod{p^{2\beta+1}},
\]
valid for $p\equiv 3\pmod{4}$. (We note in passing that this generalizes the 
congruence \eqref{3.9}.) Subtracting \eqref{6.11a} from \eqref{6.11}, we get
\[
\sum_{k=1}^{2\nu}{\rm Rem}\left(k^2\div p^{2\beta+1}\right)
=\sum_{k=1}^{\frac{p^{2\beta+1}-1}{2}}{\rm Rem}\left(k^2\div p^{2\beta+1}\right)
-\frac{p^{2\beta+1}+1}{4}.
\]
The identities \eqref{6.4} and \eqref{6.5} then follow again after some 
routine manipulations.
\end{proof}

For the proof of Lemma~\ref{lem:6.2} we require some auxiliary results, 
collected in the following lemma.

\begin{lemma}\label{lem:6.3}
Let $p$ be a prime. $(a)$ For any integers $k\geq 1$ and $\alpha\geq 1$ we have
\begin{equation}\label{6.12}
{\rm Rem}\left((kp)^2\div p^{\alpha+2}\right)
=p^2\cdot{\rm Rem}\left(k^2\div p^{\alpha}\right).
\end{equation}
$(b)$ For fixed integers $r$, $1\leq r\leq p-1$, and $\alpha\geq 0$, the 
following remainders are distinct:
\begin{equation}\label{6.13}
{\rm Rem}\left((jp+r)^2\div p^{\alpha+2}\right),\qquad
0\leq j\leq p^{\alpha+1}-1.
\end{equation}
$(c)$ For $\alpha\geq 2$ we have
\begin{equation}\label{6.14}
{\rm Rem}((2^{\alpha-1}-k)^2\div 2^\alpha)={\rm Rem}(k^2\div 2^\alpha),\quad
0\leq k\leq 2^{\alpha-1}.
\end{equation}
$(d)$ For any integers $\ell\geq 0$ and $\alpha\geq 1$ we have
\begin{equation}\label{6.15}
{\rm Rem}((2\ell+1)^2\div 2^{\alpha+2})\equiv 1\pmod{8}.
\end{equation}
\end{lemma}

\begin{proof}
(a) By division with remainder we have
\[
{\rm Rem}(k^2\div p^{\alpha})
=k^2-\left\lfloor\frac{k^2}{p^{\alpha}}\right\rfloor\cdot 2^{\alpha}
\]
for any integer $k$. Using this twice, we get
\begin{align*}
{\rm Rem}((kp)^2\div p^{\alpha+2})
&= (kp)^2-\left\lfloor\frac{(kp)^2}{p^{\alpha+2}}\right\rfloor
\cdot 2^{\alpha+2} \\
&=p^2\cdot\left(k^2-\left\lfloor\frac{k^2}{p^{\alpha}}\right\rfloor
\cdot 2^{\alpha}\right) = p^2\cdot{\rm Rem}(k^2\div p^{\alpha}),
\end{align*}
as claimed.

(b) Suppose that two of the terms in \eqref{6.13} are identical. Then 
$p^{\alpha+2}$ divides
\[
(jp+r)^2-(kp+r)^2 = ((j+k)p+2r)\cdot(j-k)\cdot p,
\]
which means that
\[
p^{\alpha+1}\left| ((j+k)p+2r)\cdot(j-k). \right.
\]
Since $1\leq r\leq p-1$, we have gcd$((j+k)p+2r,p)=1$, and so 
$p^{\alpha+1}\mid j-k$. This, by the condition in \eqref{6.13}, can only happen
when $j=k$, which completes the proof of (b).

(c) We expand $\left(2^{\alpha-1}-k\right)^2=2^{2\alpha-2}-2^\alpha\cdot k+k^2
\equiv k^2\pmod{2^\alpha}$,
where the congruence holds when $2\alpha-2\geq 2$, which is equivalent to 
$\alpha\geq 2$. This proves part~(c).

(d) We use the fact that $(2\ell+1)^2=4\ell(\ell+1)+1\equiv1\pmod{8}$. But 
also, division with remainder gives
\[
(2\ell+1)^2 = \left\lfloor\frac{(2\ell+1)^2}{2^{\alpha+2}}\right\rfloor
\cdot 2^{\alpha+2}+{\rm Rem}((2\ell+1)^2\div 2^{\alpha+2}).
\]
Since $\alpha\geq1$, this implies the congruence \eqref{6.15}.
\end{proof}

\begin{proof}[Proof of Lemma~\ref{lem:6.2}]
(a) We prove \eqref{6.6} by induction on $\beta$. When $\beta=1$, then 
$\sum_{k=1}^2{\rm Rem}(k^2\div 4)=1+0$, while on the right of \eqref{6.6} we
have $2^0(2^2-3\cdot 2^1+3)=1$; this is the induction beginning. 

For the induction step, we use symmetry (Lemma~\ref{lem:6.3}(c)) to obtain
\begin{align}
&\sum_{k=1}^{2^{2\beta+1}}{\rm Rem}(k^2\div 2^{2\beta+2})
= 2\sum_{k=1}^{2^{2\beta}}{\rm Rem}(k^2\div 2^{2\beta+2})\label{6.16}\\
&\qquad=2\sum_{\ell=1}^{2^{2\beta-1}}{\rm Rem}((2\ell)^2\div 2^{2\beta+2})
+ 2\sum_{\ell=0}^{2^{2\beta-1}-1}{\rm Rem}((2\ell+1)^2\div 2^{2\beta+2}).\nonumber
\end{align}
We deal with the sums in the second row of \eqref{6.16} separately. First, by
\eqref{6.12} we have
\begin{equation}\label{6.17}
\sum_{\ell=1}^{2^{2\beta-1}}{\rm Rem}((2\ell)^2\div 2^{2\beta+2})
= 4\sum_{\ell=1}^{2^{2\beta-1}}{\rm Rem}(\ell^2\div 2^{2\beta})
= 2^{2\beta}\left(2^{2\beta}-3\cdot 2^{\beta}+3\right),
\end{equation}
where we have used \eqref{6.6} as induction hypothesis.

To deal with the second sum on the right of \eqref{6.16}, we note that by parts 
(b) and (d) of Lemma~\ref{lem:6.3} the $2^{2\beta-1}$ summands are distinct 
positive integers of the form $8j+1$, $0\leq j\leq 2^{2\beta-1}-1$. Since
$8j+1 < 8\cdot 2^{2\beta-1} = 2^{2\beta+2}$, these summands are not reduced
modulo $2^{2\beta+2}$, and we have
\begin{align}
\sum_{\ell=0}^{2^{2\beta-1}-1}{\rm Rem}((2\ell+1)^2\div 2^{2\beta+2})
&= \sum_{\ell=0}^{2^{2\beta-1}-1}(8j+1) \label{6.18}\\
&= 2^{2\beta-1}+8\cdot\frac{1}{2}\left(2^{2\beta-1}-1\right)2^{2\beta-1}\nonumber \\
&= 2^{2\beta-1}\left(2^{2\beta+1}-3\right).\nonumber
\end{align}
Finally, substituting \eqref{6.17} and \eqref{6.18} into \eqref{6.16}, we get 
after some straightforward manipulations,
\[
\sum_{k=1}^{2^{2\beta+1}}{\rm Rem}(k^2\div 2^{2\beta+2})
= 2^{2\beta}\left(2^{2\beta+2}-3\cdot 2^{\beta+1}+3\right).
\]
This completes the proof of \eqref{6.6} by induction. The proof of \eqref{6.7}
is analogous; we leave the details to the reader.

(b) We begin with \eqref{6.9}, using again induction on $\beta$. For $\beta=0$,
\eqref{6.9} reduces to 
\begin{equation}\label{6.18a}
\sum_{k=1}^{\frac{p-1}{2}}{\rm Rem}(k^2\div p)=\frac{1}{4}p(p-1),
\end{equation}
which is true by \eqref{2.24} and \eqref{2.25}. We now assume that \eqref{6.9}
holds for some $\beta\geq 0$. Using the fact that for any integer $\alpha\geq 1$
we have the symmetry relation
\[
\left(p^{\alpha}-k\right)^2 = p^{2\alpha}-2kp^{\alpha}+k^2
\equiv k^2\pmod{p^{\alpha}},
\]
we can write
\begin{align}
2\sum_{k=1}^{\frac{p^{2\beta+3}-1}{2}}{\rm Rem}\left(k^2\div p^{2\beta+3}\right)
&=\sum_{k=1}^{p^{2\beta+3}-1}{\rm Rem}\left(k^2\div p^{2\beta+3}\right)\label{6.19}\\
&=\sum_{r=0}^{p-1}\sum_{j=0}^{p^{2\beta+2}-1}
{\rm Rem}\left((jp+r)^2\div p^{2\beta+3}\right).\nonumber
\end{align}
We will now evaluate the inner sum on the right of \eqref{6.19}. When $r=0$,
then \eqref{6.12} gives
\begin{equation}\label{6.20}
\sum_{j=0}^{p^{2\beta+2}-1}{\rm Rem}\left((jp)^2\div p^{2\beta+3}\right)
=p^2\sum_{j=0}^{p^{2\beta+2}-1}{\rm Rem}\left(j^2\div p^{2\beta+1}\right).
\end{equation}
We split $j$ as $j=\ell\cdot p^{2\beta+1}+s$, with $0\leq\ell\leq p-1$ and
$0\leq s\leq p^{2\beta+1}-1$ and note that
\[
{\rm Rem}\left((\ell\cdot p^{2\beta+1}+s)^2\div p^{2\beta+1}\right)
={\rm Rem}\left(s^2\div p^{2\beta+1}\right).
\]
Then
\begin{align*}
\sum_{j=0}^{p^{2\beta+2}-1}{\rm Rem}\left(j^2\div p^{2\beta+1}\right)
&=\sum_{\ell=0}^{p-1}\sum_{s=0}^{p^{2\beta+1}-1}
{\rm Rem}\left(s^2\div p^{2\beta+1}\right) \\
&=2p\sum_{s=0}^{\frac{p^{2\beta+1}-1}{2}}
{\rm Rem}\left(s^2\div p^{2\beta+1}\right) \\
&= 2p\cdot\frac{1}{4}p^{3\beta+1}\left(p^{\beta+1}-1\right),
\end{align*}
where we have used symmetry and then \eqref{6.9} as induction hypothesis.
With \eqref{6.20} we now get
\begin{equation}\label{6.21}
\sum_{j=0}^{p^{2\beta+2}-1}{\rm Rem}\left((jp)^2\div p^{2\beta+3}\right)
= \frac{1}{2}p^{3\beta+4}\left(p^{\beta+1}-1\right).
\end{equation}
Next we consider $1\leq r\leq p-1$ in the inner sum in \eqref{6.19}; for 
greater generality, we consider integers $\alpha\geq 0$. To simplify notation,
we set
\begin{equation}\label{6.22}
a_r := {\rm Rem}(r^2\div p),
\end{equation}
and note that $(jp+r)^2\equiv a_r\pmod{p}$. By Lemma~\ref{lem:6.3}(b), the
terms ${\rm Rem}((jp+r)^2\div p^{\alpha+2})$ are all distinct for
$0\leq j\leq p^{\alpha+1}-1$, they are all congruent to $a_r\pmod{p}$, and lie
between 0 and $p^{\alpha+2}$. So they are $a_r+jp$, 
$0\leq j\leq p^{\alpha+1}-1$, in some order. Hence, for $1\leq r\leq p-1$,
\begin{align*}
S_r&:=\sum_{j=0}^{p^{\alpha+1}-1}{\rm Rem}((jp+r)^2\div p^{\alpha+2})
=\sum_{j=0}^{p^{\alpha+1}-1}(a_r+jp) \\
&= p^{\alpha+1}a_r + \frac{p}{2}\left(p^{\alpha+1}-1\right)p^{\alpha+1}
= p^{\alpha+1}\left(a_r+ \frac{p}{2}\left(p^{\alpha+1}-1\right)\right).
\end{align*}
Next we use the fact that by \eqref{6.22} and \eqref{6.18a} we have
$\sum_{r=1}^{p-1}a_r=\frac{1}{2}p(p-1)$, so that for $\alpha\geq 0$,
\begin{equation}\label{6.23}
\sum_{r=1}^{p-1}S_r
=p^{\alpha+1}\left(\frac{p(p-1)}{2}
+(p-1)\cdot\frac{p}{2}\left(p^{\alpha+1}-1\right)\right)
=\frac{p-1}{2}\cdot p^{\alpha+3}.
\end{equation}
Finally, upon setting $\alpha=2\beta+1$ in \eqref{6.23} and substituting this
and \eqref{6.21} into \eqref{6.19}, we get
\[
\sum_{k=1}^{\frac{p^{2\beta+3}-1}{2}}{\rm Rem}\left(k^2\div p^{2\beta+3}\right)
= \frac{1}{4}\cdot p^{3\beta+4}\left(p^{\beta+2}-1\right).
\]
Comparing this with \eqref{6.9}, we see that the proof by induction is complete.

To prove \eqref{6.8}, we first note that \eqref{6.23} with $\alpha=0$ gives
\[
\sum_{r=1}^{p-1}\sum_{j=0}^{p-1}{\rm Rem}((jp+r)^2\div p^2)
= \frac{p-1}{2}\cdot p^3,
\]
while for $r=0$ we have ${\rm Rem}((jp)^2\div p^2)=0$. Hence
\[
\sum_{k=1}^{p^2-1}{\rm Rem}(k^2\div p^2) = \frac{1}{2}\cdot p^3(p-1),
\]
which is the induction beginning for $\beta=1$ if we take symmetry into 
account. The remainder of the proof of \eqref{6.8} is completely analogous to
that of \eqref{6.9}.

(c) The proofs of the identities \eqref{6.10} and \eqref{6.11} are similar to
those of \eqref{6.8} and \eqref{6.9}, and we leave the details to the 
interested reader. However, the case $p=3$ requires some attention. Rather than
dealing with the details of Dirichlet's class number formula \eqref{3.2}, we
verify that \eqref{3.7} holds for $p=3$ if $h(-3)$ is replaced by 
$\frac{1}{3}=h^*(-3)$. Similarly, \eqref{4.11} holds for $p=3$ with $h(-3)$
replaced by $h^*(-3)$.

This last identity is then the induction beginning, with $\beta=0$, in the
proof of \eqref{6.11}. One other difference between the proofs of \eqref{6.11}
and \eqref{6.9} is that for summing the terms $a_r$ we need to use \eqref{4.11}
again, with the appropriate change for $p=3$. The proof of \eqref{6.10} is
again similar.
\end{proof}

\section{Some conjectures and remarks}\label{sec:7}

The results in this paper so far give rise to the question: What can we say
about $f(n)$ when $n$ has two or more distinct prime factors, at least one of
which is of the form $q\equiv 3\pmod{4}$? With Propositions~\ref{prop:5.1} 
and~\ref{prop:6.1} in mind, the next step would be the case 
$n=p^{\alpha}q^{\beta}$, where one or both of $p$ and $q$ are $\equiv 3\pmod{4}$
and $\alpha, \beta\geq 1$.

Extensive computations led us to formulate the following conjectures. For the
sake of simplicity and in view of Proposition~\ref{prop:2.1}, we state some
conjectured identities only for the sum
\[
S_n:=\frac{1}{n}\sum_{k=1}^{n-1}{\rm Rem}(k^2\div n).
\]
As we did in Proposition~\ref{prop:6.1}, we set $h^*(-p)=h(-p)$ when 
$p\geq 7$, and $h^*(-3)=1/3$.

\begin{conjecture}\label{conj:7.1}
Let $n=p^{\alpha}q^{\beta}$ with $p\equiv 1\pmod{4}$ and $q \equiv 3\pmod{4}$
both primes and $\alpha, \beta\geq 1$. Then
\begin{align}
S_n &= 
\frac{n-p^{\lfloor\frac{\alpha}{2}\rfloor}q^{\lfloor\frac{\beta}{2}\rfloor}}{2}
-\frac{q^{\lfloor\frac{\beta+1}{2}\rfloor}-1}{q-1}\label{7.1}\\
&\quad\times\left(\left(\frac{p^{\lfloor\frac{\alpha+2}{2}\rfloor}-1}{p-1}
-\frac{p^{\lfloor\frac{\alpha}{2}\rfloor}-1}{p-1}\left(\frac{p}{q}\right)\right)
h^*(-q)-\frac{p^{\lfloor\frac{\alpha+1}{2}\rfloor}-1}{p-1}h^*(-pq)\right).\nonumber
\end{align}
\end{conjecture}

\begin{conjecture}\label{conj:7.2}
Let $n=p^{\alpha}q^{\beta}$ with $p\equiv q\equiv 3\pmod{4}$ distinct primes and $\alpha, \beta\geq 1$. Then
\begin{align}
S_n =
\frac{n-p^{\lfloor\frac{\alpha}{2}\rfloor}q^{\lfloor\frac{\beta}{2}\rfloor}}{2}
&-\left(\frac{q^{\lfloor\frac{\beta}{2}\rfloor}-1}{q-1}\left(\frac{p}{q}\right)
+\frac{q^{\lfloor\frac{\beta}{2}\rfloor+1}-1}{q-1}\right)\cdot
\frac{p^{\lfloor\frac{\alpha+1}{2}\rfloor}-1}{p-1}h^*(-p) \label{7.2}\\
&-\left(\frac{p^{\lfloor\frac{\alpha}{2}\rfloor}-1}{p-1}\left(\frac{q}{p}\right)
+\frac{p^{\lfloor\frac{\alpha}{2}\rfloor+1}-1}{p-1}\right)\cdot
\frac{q^{\lfloor\frac{\beta+1}{2}\rfloor}-1}{q-1}h^*(-q). \nonumber
\end{align}
\end{conjecture}

In the special case $\alpha=\beta=1$, \eqref{7.2} becomes
\begin{equation}\label{7.3}
S_{pq} = \frac{pq-1}{2} - h^*(-p) - h^*(-q)
\qquad(p\equiv q\equiv 3\pmod{4}\quad\hbox{distinct}),
\end{equation}
and Proposition~\ref{prop:2.1} and \eqref{1.4} give, after some effort,
\begin{equation}\label{7.4}
f(pq) = -\frac{h^*(-p)+h^*(-q)}{2}
\qquad(p\equiv q\equiv 3\pmod{4}\quad\hbox{distinct}).
\end{equation}

The next conjecture concerns all squarefree odd positive integers and can be
seen as the opposite extreme of Proposition~\ref{prop:6.1}.

\begin{conjecture}\label{conj:7.3}
Let $n$ be an odd squarefree positive integer.\\
$(a)$ If $n\equiv 1\pmod{4}$, then
\begin{equation}\label{7.5}
f(n) = -\frac{1}{2}\sum_{\substack{d|n\\d\equiv 3 (\rm{mod}\,4)}}h^*(-d).
\end{equation}
$(b)$ If $n\equiv 3\pmod{4}$, then
\begin{equation}\label{7.6}
f(n) = \frac{1-n}{4} 
-\frac{1}{2}\sum_{\substack{d|n\\d\equiv 3 (\rm{mod}\,4)}}h^*(-d).
\end{equation}
\end{conjecture}
Special cases of this conjecture include Proposition~\ref{prop:1.1}, the first 
identity of \eqref{5.1a}, and \eqref{7.4}.
The two identities \eqref{7.5} and \eqref{7.6} can be written jointly as
\begin{equation}\label{7.6a}
\frac{1-n}{4}\cdot\delta(n)-f(n)=\sum_{d|n}h^*(-d)\cdot\frac{\delta(d)}{2},
\end{equation}
where $n$ is a squarefree odd positive integer and
\begin{equation}\label{7.6b}
\delta(n)=\begin{cases}
1 & \hbox{if}\;n\equiv 3\pmod{4},\\
0 & \hbox{otherwise}.
\end{cases}
\end{equation}
We now extend the conjectured identity \eqref{7.6a} to all positive integers
$n$ by defining $\mathcal{F}(n)$ to be the left-hand side of \eqref{7.6a} when
$n$ is squarefree and odd, and to be the right-hand side of \eqref{7.6a} 
otherwise. Using the M\"obius inversion formula, we then get
\begin{equation}\label{7.6c}
h^*(-n)\cdot\frac{\delta(n)}{2}
=\sum_{d|n}\mu(\tfrac{n}{d})\mathcal{F}(d).
\end{equation}
When $n$ is squarefree and odd, then so are all divisors $d$ of $n$ 
and for $\mathcal{F}(d)$ we can
use the left-hand side of \eqref{7.6a}. The identity \eqref{7.6c} therefore
shows that the truth of Conjecture~\ref{conj:7.3} implies that the following
is also true.

\begin{conjecture}\label{conj:7.4}
Let $n$ be an odd squarefree positive integer. Then
\begin{equation}\label{7.6d}
h^*(-n)\cdot\delta(n)
=\sum_{d|n}\mu(\tfrac{n}{d})\left(\frac{1-d}{2}\cdot\delta(d)-2f(d)\right),
\end{equation}
where $\delta(n)$ is as defined in \eqref{7.6b}.
\end{conjecture}

This means that, conjecturally, all class numbers $h^*(-n)$ for odd squarefree
integers $n>0$ can be written in terms of the sum \eqref{2.1}. 
Proposition~\ref{prop:1.2} is a special case.

\medskip
Finally, if we plot $f(n)$, as defined in \eqref{1.4}, some striking 
distributions become apparent; see Figure~1. Upon closer examination, this 
leads to the following conjecture.

\bigskip
\begin{center}
\includegraphics[scale=0.4]{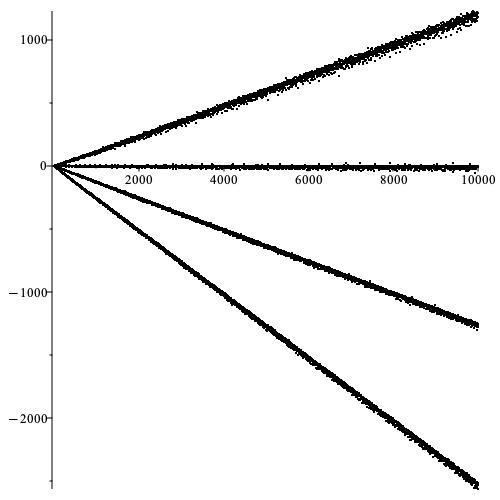}

\medskip
{\bf Figure~1}: $f(n)$ for $1\leq n\leq 10\,000$.
\end{center}
\medskip

\begin{conjecture}\label{conj:7.5}
The sequence $f(n)$ satisfies the following limits:
\begin{align}
\lim_{n\to\infty}\frac{f(4n)}{4n} & = \frac{1}{8},\label{7.7}\\
\lim_{n\to\infty}\frac{f(4n+1)}{4n+1} & = 0,\label{7.8}\\
\lim_{n\to\infty}\frac{f(4n+2)}{4n+2} & = -\frac{1}{8},\label{7.9}\\
\lim_{n\to\infty}\frac{f(4n+2)}{4n+3} & = -\frac{1}{4}.\label{7.10}
\end{align}
\end{conjecture}

The four limits in Conjecture~\ref{conj:7.5} are supported by the following 
proven or conjectured identities:

\begin{enumerate}
\item[(7.11):] \eqref{4.4}, \eqref{6.1}, \eqref{6.2};
\item[(7.12):] \eqref{1.5}, \eqref{5.1a}(i), \eqref{6.3}, \eqref{6.4}, 
\eqref{7.4}, \eqref{7.5};
\item[(7.13):] \eqref{5.1a}(ii), \eqref{4.3};
\item[(7.14):] \eqref{1.6}, \eqref{6.5}, \eqref{7.6}.
\end{enumerate}

In several of these cases we need the well-know fact that 
$h^*(-n) = O(\sqrt{n}\log{n})$; see, e.g., \cite[p.~138]{Co1}.


\end{document}